\newif\ifpgf
\documentclass[leqno,onefignum,onetabnum,onethmnum]{siamltex1213} \pgftrue

\title{The Scaling, Splitting and Squaring Method for the Exponential of Perturbed Matrices}

\author{
Philipp Bader\footnote{Instituto de Matem\'atica Multidisciplinar,
  Universitat Polit\`ecnica de Val\`{e}ncia, E-46022  Valencia, Spain.
 {\tt phiba@imm.upv.es}}, 
Sergio Blanes\footnote{Instituto de Matem\'atica Multidisciplinar,
  Universitat Polit\`ecnica de Val\`{e}ncia, E-46022  Valencia, Spain.
 {\tt serblaza@imm.upv.es}}, AND
Muaz Seydao\u{g}lu\footnote{Department of Mathematics, Faculty of Art and Science,
Mu\c{s} Alparslan University, 49100 Mu\c{s}, Turkey.
{\tt muasey@imm.upv.es}}
}

\usepackage{graphicx}
\usepackage{mathdots}
\usepackage{array}

		\usepackage{pgfplots}
		\usetikzlibrary{plotmarks}
		\usetikzlibrary{external}
		\pgfplotsset{compat=1.3}
		\newlength\figurewidth 
		\newlength\figureheight
\makeatletter
\def\Ddots{\mathinner{\mkern1mu\raise\p@
\vbox{\kern7\p@\hbox{.}}\mkern2mu
\raise3\p@\hbox{.}\mkern2mu\raise5\p@\hbox{.}\mkern1mu}}
\makeatother
\usepackage{booktabs}
\newcommand{\red}[1]{}

\usepackage{amsmath,amssymb}
\def\cost{c}

\DeclareMathOperator\bch{bch}
\def\en{\text{\sc{e}-}}
\def\ep{\text{\sc{e}}}

\def\e{e}
\def\id{I}
\newcommand\defn[1]{\textit{#1}}
\def\eps{\varepsilon}
\def\cB{\mathcal{B}}
\def\cD{\mathcal{D}}
\def\cO{\mathcal{O}}
\def\C{\mathbb{C}}
\def\R{\mathbb{R}}
\def\beq{\begin{equation}}
\def\eeq{\end{equation}}

\begin{document}

\maketitle

\begin{abstract}
We propose splitting methods for the computation of the exponential of perturbed matrices which can be written as the sum $A=D+\varepsilon B$ of a sparse and efficiently exponentiable matrix $D$ with sparse exponential $e^D$ and a dense matrix $\varepsilon B$ which is of small norm in comparison with $D$.
The predominant algorithm is based on scaling the large matrix $A$ by a small number $2^{-s}$, which is then exponentiated by efficient Pad\'e or Taylor methods and finally squared in order to obtain an approximation for the full exponential.
In this setting, the main portion of the computational cost arises from dense-matrix multiplications and we present a modified squaring which takes advantage of the smallness of the perturbed matrix $B$ in order to reduce the number of squarings necessary.
Theoretical results on local error and error propagation for splitting methods are complemented with numerical experiments and show a clear improvement over existing methods when medium precision is sought.
\end{abstract}



\pagestyle{myheadings}
\thispagestyle{plain}
\markboth{BADER, BLANES, SEYDAOGLU}{SCALING, SQUARING AND SPLITTING}

\section{Introduction}
The efficient computation of matrix exponentials has been extensively considered in the literature and the \defn{scaling and squaring method} is perhaps the most widely used method for matrices of dimension $n\times n$ with $n$ as large as a few hundred (see \cite{higham09tsa,moler03ndw,sidje98eas} and references therein). 
	For example, Matlab and Mathematica compute numerically the exponential of matrices using this method where highly efficient algorithms for general matrices exist \cite{almohy09ans,higham05tsa,higham09tsa,higham10cma}. 

Given $A \in\C^{n\times n}$, the method is based on the property
\begin{equation}
	\e^A = \left( \e^{A/2^s}  \right)^{2^s} = 
	\underbrace{\left( \cdots \left.\left( \e^{A/2^s}  \right)^{2}\right.^{\Ddots}\right)^{2} }_{s-\text{times}},
\end{equation}
where typically $\e^{A/2^s}$ is replaced by a polynomial approximation (e.g. a $m$th-order Taylor method, $T_m(A/2^s)$) or a rational approximation (e.g. an $2m$th-order diagonal Pad\'e method, $r_{2m}(A/2^s)$) \cite{higham09tsa,higham10cma,sastre14aae}. 
The optimal choice of both $s$ and the algorithms to compute $\e^{A/2^s}$ usually depend on the value of $\|A\|$ and the desired tolerance, and
have been deeply analyzed.


The computational cost, $c(\cdot)$, is usually measured by the number of matrix--matrix products, so $c(\e^A) = s+c(\e^{A/2^s})$, where $c(\e^{A/2^s})$ has to be replaced by the cost of its numerical approximation, e.g.  $c(T_{m}(A/2^s))$ or  $c(r_{2m}(A/2^s))$. Given a tolerance, one has to look for the scheme which provides such accuracy with the minimum number of products
 (see  \cite{higham09tsa,higham10cma} and references therein).

In some cases, if the matrix $A$ has a given structure, more efficient methods can be obtained \cite{celledoni00ate,celledoni01mft} . For example, to compute the exponential of upper or lower triangular matrices, in \cite{almohy09ans} the authors show that it is advantageous to exploit the fact that the diagonal 
elements of the exponential are exactly known. 
It is then more efficient to replace the diagonal elements obtained using e.g. Taylor or Pad\'e approximations by the exact solution before squaring the matrix (this technique can also be extended to the  first super (or sub-)diagonal elements).

On the other hand, in many cases the matrix $A$ can be considered as a small perturbation of a sparse 
matrix $D$, i.e., $A=D+B$ with $\|B\|<\|D\|$ (and frequently $\|B\|\ll\|D\|$) where $\e^D$ is sparse and exactly solvable (or can be accurately and cheaply approximated numerically), and $B$ is a dense matrix. 
This is the case, for example, if $D$ is diagonal (or block diagonal with small matrices along the diagonal), or if it is diagonalizable using only a few elementary transforms. This is also the case, for example, if $n=2k$ and
\[
  D=\left(\begin{array}{cc} 0 & I \\ -\Omega^2 & 0 \end{array} \right)
\]
where $I$ is the $k\times k$ identity matrix and $\Omega$ is a diagonal matrix where $\e^D$ is also an sparse and trivial to compute matrix. This problem can be originated from a semidiscretization of a hyperbolic PDE or from a set of $k$ linearly coupled oscillators.

As a motivational example, let us consider the linear time-dependent system of differential equations
\[
  \frac{d}{dt} X = M(\eps t) X, \qquad X(t_0)=X_0 \in\C^{n\times n}
\]
with $M \in\C^{n\times n}$ and $|\eps|\ll 1$, i.e., $M(\eps t)$ evolves adiabatically with the variable $t$. Suppose that $M(\eps t)$ is instantaneously diagonalizable, i.e., $M(\eps t)=Q(\eps t)D(\eps t) Q^{-1}(\eps t)$ with $D$ a diagonal matrix. 
Then, we can consider what it is usually called the adiabatic picture in quantum mechanics (if $M$ is a skew-Hermitian matrix), i.e., the change of variables, $X=Q(\eps t)Y$ where $Y$ is the solution of the differential equation
\[
  \frac{d}{dt} Y = \left(D - Q^{-1}\frac{d}{dt}Q\right) Y, \qquad 
	Y(t_0)=Q^{-1}(\eps t_0)X_0.
\]
A second order method in the time step $h$ which advances the solution from $t_i$ to $t_i+h$, where $Y_i\approx Y(t_i)$, is given by
\begin{equation}\label{eq.midpoint}
  Y_{i+1} = e^{h\left(D_{1/2} + \eps B_{1/2}\right)} Y_{i},
\end{equation}
where
\[
 D_{1/2}= D\left(\eps\left(t_{i+1/2}\right)\right) , \qquad
 \eps B_{1/2}= - Q^{-1}\left(\eps \left(t_{i+1/2}\right)\right)
\frac{d}{dt}Q\left(\eps \left(t_{i+1/2}\right)\right),
\]
with $t_{i+1/2}=t_i+\frac{h}2$. Notice that $\eps B_{1/2}$ is, in general, a dense matrix with a small norm (proportional to $\eps$) due to the term $\frac{d}{dt}Q(\eps t)$.

It is then natural to look for methods that approximate the  exponential \eqref{eq.midpoint} at a low computational cost while providing sufficient accuracy. 
Notice that in most cases in practice it is not necessary to approximate the exponential up to round-off accuracy since the model/method itself does not reproduce the exact solution within round-off precision. However, the preservation of qualitative properties (e.g. orthogonality, symplecticity, unitarity, etc.) is in some cases of great interest \cite{iserles00lgm}.

The aim of this work is the exploration of new and more efficient algorithms which take advantage of the fact that $\e^D$ is sparse and known at a cheap computational cost and that $B$ has a small norm. 
The schemes we analyze in continuation are based on splitting and composition techniques tailored for this particular problem. 

For clarity in the presentation, we take the partition $s=s_1+s_2$, we set $h=2^{-s_2}, N=1/h=2^{s_2}$ and replace $B$ by $\eps B$ with $\|B\|\sim\|D\|$, and we propose a new recursive procedure that we refer as Modified Squaring 
\begin{equation} \label{eq.Recursive}
	X_0 = e^{b h \eps B}, \qquad
	X_k = X_{k-1}e^{a_k h D}X_{k-1}, \qquad k=1,\ldots, s_1
\end{equation}
and $Y_{s_1} = e^{a_{s_1+1}hD}X_{s_1}e^{a_{s_1+1}hD}$ where $b=1/2^{s_1}$ and the parameters $a_k$ will be chosen properly to improve accuracy.  The total cost is
\[
 c(Y_{s_1}^{s_2}) = s_1 + s_2 + c(e^{b h \eps B})
\]
where $c(e^{b h \eps B})=c(e^{\eps B/2^s})$ is the cost to approximate this exponential. Since $\|h\eps B\|$ is very small, a low-order diagonal Pad\'e approximation can provide sufficient accuracy (for most problems it will suffice just to consider $r_2$ or $r_4$ which only require one inversion or one inversion and one product, or even a low-order Taylor approximation can also be used).

The choice $s_1=0$ corresponds to the Leapfrog or Strang method,
\begin{equation} \label{eq.leapfrog}
	\e^{h(D+\eps B)} \approx 	\e^{hD/2}	\e^{h\eps B}	\e^{hD/2},
\end{equation}
where, as already mentioned, $\e^{hD/2}$ can be accurately and cheaply computed.

More accurate methods can be obtained using a general composition
\begin{equation}\label{eq.splitting}
S^{[m]}_p =	\prod_{i=1}^m\e^{ha_iD}	\e^{hb_i \eps B}	\approx \e^{h(D+\eps B)},
\end{equation}
where the coefficients $a_i,b_i$ are chosen such that $S_p^{[m]}$  is an approximation to the exact solution up to a given order, $p$, in the parameter $h$, i.e. $S_p^{[m]}=\e^{h(D+\eps B)}+\cO(h^{p+1})$.
However, to get efficient methods it is crucial to reduce the computational cost.
Since the cost is dominated by the exponentials $\e^{hb_i \eps B}$, it is advisable to reuse as many exponentials as possible, e.g., letting $b_i=1/m$, only one exponentiation is necessary. 
However, this class of methods has some limitations since for orders greater than 2, at least one of the coefficients $a_i$ and one of the $b_i$ must be negative and thus might jeopardize the re-utilization of the exponentials. 
However, for small perturbations, very accurate results can still be obtained with positive coefficients.

In the particular situation when $A\in \C^{n\times n}$, complex coefficients, $a_i\in \C$, can be used without increasing the computational cost, and then fourth-order methods with all $b_i$ real and equal are achievable. The proposed recursive algorithm (\ref{eq.Recursive}) corresponds to a particular case of an splitting method where the cost has been reduced while still leaving some free parameters for optimisation.

In this work, we assume that the product $B^2$ requires $\cO\left({{n}^3}\right)$ operations but $DB$ requires only $\cO\left({k{n}^2}\right)$ with $k\ll {n}$ (e.g. $c(B^2)=1, \ c(DB)=\delta$, with $\delta\ll 1$). 
Then, the commutator $\eps[D,B]=\eps(DB-BD)$ can be computed at considerably smaller cost than the product of two dense matrices while retaining a small norm due to the factor $\eps$. 
It then makes sense to consider the recursive algorithm (\ref{eq.Recursive}) where the exponential $e^{b h \eps B}$ is replaced by 
\begin{equation} \label{eq.leapfrogMod}
		\e^{bh\eps B+ \alpha h^3\eps [A,[A,B]]}	
\end{equation}
whose computational cost is similar,
 but more accurate results can be obtained if the scalar parameter $\alpha$ is properly chosen. 
Further exploiting this approach leads to the inclusion of the term $\beta h^5\eps [A,[A,[A,[A,B]]]]$ in the central exponential, which again, for an appropriate choice of the parameter $\beta$, decreases the error at a similar computational cost. The analysis presented in this work is also extended to the case in which not all parameters $b_i$ are taken equal.

 This paper is organized as follows: 
Section~2 considers the computational cost of Pad\'e and Taylor methods as well as the cost of all operations involved in the splitting schemes analyzed in this work in order to develop new algorithms which minimize the whole cost. 
In Section~3 we analyze the algebraic structure of the different families of methods considered to obtain the order conditions to be satisfied by the coefficients. 
In Section~4 we propose a recursive algorithms to minimize the cost of the methods and we build new methods. An error analysis is carried in Section~5 and Section~6 illustrates  the performance of the methods on several numerical examples. Finally, Section~7 presents the conclusions and the appendix collects, for completeness, several new families of splitting methods which have also been analyzed.

\section{Computational cost of matrix exponentiation}

\subsection{Computational cost of Taylor and Pad\'e methods}

We first review the computational cost of the optimized Taylor and Pad\'e methods which are used in the literature and that are used as reference in the numerical examples.

\paragraph{Taylor methods}

We use the Paterson-Stockmeyer scheme (see \cite{higham08fom,higham10cma,paterson73otn}) to evaluate $T_m=\sum_{k=0}^mA^n/n!$  which minimize the required number of products. 

From the Horner-scheme-like computation, 
given a number of matrix products $2k$, the maximal attainable order is $m=(k+1)^2$.
In \cite{higham10cma}, it is indicated that the optimal choice for most cases corresponds to $k=3$, i.e. order $m=16$ 
with just 6 products given by: $A^2=AA, \ A^3=A^2A, \ A^4=A^2A^2$ and
\[
  T_{16}(A) = g_0+(g_1+(g_2+(g_3+g_4 A^4) A^4) A^4) A^4,
\]
where $g_i$ are linear combinations of already computed matrices,
$g_i=\sum_{k=0}^4 c_{i,k} A^k$, with $c_{i,k}=1/(4i+k)!$ for $i=0,1,2,3$ and $g_4=\id/16$ proportional to the identity (matrix).

\paragraph{Diagonal Pad\'e methods}
Diagonal Pad\'e methods are given by the rational approximant
\begin{equation}\label{eq.24}
		r_{2m}(A )=\frac{p_{m}(A )}{p_{m}(-A )},
\end{equation}
provided the polynomials $p_{m}$ are generated by the recurrence
\begin{align}
		p_{0}(A ) &= \id, \ \qquad  p_{1}(A )=2\id+A  \nonumber \\
		p_{m}(A ) &=2(2m-1)p_{m-1}(A )+A^{2}p_{m-2}(A ).  \label{eq.RecPade}
\end{align}
Moreover,
$r_{2m}(A )= \e^{A }+\mathcal{O}(A^{2m+1})$, whereas
for $m=1,2$ we have
\begin{equation}\label{eq.Pade24}
r_{2}(A ) = \frac{\id+A/2}{ \id-A/2}, \qquad  \qquad 
r_{4}(A ) = \frac{\id+A/2+A ^{2}/12}{ \id-A/2+A ^{2}/12}.
\end{equation}
The recursive algorithm (\ref{eq.RecPade}) is, however, not an efficient way to compute $r_{2m}(A )$. 
For example, the method $r_{26}(A)$ is considered among the optimal choices (with respect to accuracy and computational cost) of diagonal Pad\'e methods when round off accuracy is desired and $\|A\|$ takes relatively large values. 
The algorithm to compute  it is given by
\begin{equation}
	(-u_{13}+v_{13}) r_{26}(A) = (u_{13}+v_{13}),
\end{equation}
with 
\begin{align*}
	u_{13} & =  A[A_6(b_{13}A_6 + b_{11}A_4 + b_9A_2) + b_7A_6 + b_5A_4 + b_3A_2 + b_1 \id ], \\
	v_{13} & =  A_6(b_{12}A_6 + b_{10}A_4 + b_8A_2) + b_6A_6 + b_4A_4 + b_2A_2 + b_0\id,
\end{align*}
where $A_2=A^2, A_4=A_2^2, A_6=A_2A_4$.
Written in this form, it is evident that only six matrix multiplications and one inversion are required.  
In a similar way, the method $r_{10}(A)$, which will be used in this work, only requires 3 products and one inversion.

\subsection{Computational cost of splitting methods}

Recall that we are considering a sparse and sparsely exponentiable matrix $D$, while $B$ is a dense matrix and responsible for the numerical complexity. 
In order to build competitive algorithms, it is important to analyze - under these assumptions - the computational cost of all operations involved in the different classes of splitting and composition methods.

Let $X,Y$ be two dense $n\times n$ matrices and denote by $c(\cdot)$ the cost of the operations in brackets as the number of matrix--matrix products of dense matrices, e.g., $c(X Y)=1$ and $c(X+ Y)=\delta$, with $\delta\ll 1$, thereby neglecting operations with a lower complexity in the number of operations. According to this criterion, we derive Table~\ref{tab.cost}, where the dominant terms are highlighted in boldface (the cost for the inverse of a matrix is taken as 4/3 the cost of a matrix-matrix product).
\begin{table}[!ht]
\begin{center}\footnotesize
\begin{tabular}{lll}
\toprule
 & Operation& Effort\\[1mm]
Sum & $c(D+ D)\approx 0$  &${\cal O}(k\, n)$, with $k\ll n$\\
& $c(X+ Y)=\delta$&${\cal O}(n^2)$\\[.5mm]
Product & $c(XY)={\bf 1}$&${\cal O}(n^3)$\\
 & $c(D D)=0$&${\cal O}(k^2\, n)$\\
	&$c(D X)=k\delta$&${\cal O}(k\, n^2)$\\

Inversion& $c(X^{-1} Y)={\bf 1+\frac13}$&$c(X^{-1} Y)=\frac43c(X Y)$\\[.5mm]
Commutation
&$c([D,X]) = c(D X-X D)=2k\delta$ & ${\cal O}(k\, n^2)$\\
&$c([D,[D,\ldots,[D,X]\cdots]])=2rk\delta$&${\cal O}(k\, n^2)$\\[.5mm]
Exponentiation & $c(\e^D)=wk\delta$&${\cal O}(k^2\, n)$\\
& $c(r_{2}(X))={\bf 1+\frac13}$  & ${\cal O}(n^3)$\\
& $c(r_{4}(X))={\bf 2+\frac13}$  & ${\cal O}(n^3)$\\
\bottomrule
\end{tabular}
\end{center}
\caption{\label{tab.cost}Computational cost of matrix operations for the sparse and sparsely exponentiable matrix $D$ and arbitrary dense matrices $X,Y\in\mathbb{C}^{n\times n}$. The factor $w$ in $c(e^D)$ is assumed to be small, $w\ll1$.}
\end{table}

Based on this analysis, we examine the splitting method \eqref{eq.splitting} to identify the computationally relevant aspects. 
 In this work we assume $\delta\ll 1$ and in our computations we will take $\delta=0$ for simplicity.
First, we have to choose how to approximate the exponentials $\e^{h\eps b_i B}$ taking into account that
\begin{eqnarray}
	r_2(h\eps b_i B)&=&  \e^{h\eps b_i B} +{\cal O}(h^3 \eps^3)	 \label{eq.Pade2} , \\ 
	r_4(h\eps b_i B)&=&  \e^{h\eps b_i B} +{\cal O}(h^5 \eps^5)	 \label{eq.Pade4} .
\end{eqnarray}
A rough estimate for the composition  \eqref{eq.splitting}, assuming all coefficients $b_i$ different, and taking into account the cost shown in Table~\ref{tab.cost}, we have
\[
 c(S^{[m]}_p,r_2)=m\frac43+m-1=\frac73 m -1, \qquad
 c(S^{[m]}_p,r_4)=m\frac73+m-1 = \frac{10}{3}m-1,
\]
where $c(S^{[m]}_p,r_i)$ denotes the cost of the method $S^{[m]}_p$ when the exponentials $\e^{\eps B}$ are approximated by $r_i(\eps B)$.
Repeating the coefficients $b_i$, i.e., $b_i=1/m, \ i=1,\ldots,m$, the computational cost can be reduced considerably, in this case, one gets
\[
 c(S^{[m]}_p,r_2)=\frac43+(m-1)=m+\frac13, \qquad
 c(S^{[m]}_p,r_4)=m+\frac43.
\]
Further simplifications are applicable and will be discussed in Sect.~\ref{sec:splitting}.

\section{The Lie algebra of perturbed systems: $(p_1,p_2)$ methods}
Following the terminology of \cite{mclachlan95cmi}, we introduce a modified error concept which is suitable for the near-integrable structure of the matrix $A$ at hand.

Letting $S^{[m]}_p$ be a $p$th-order $m$-stage consistent ($\sum_i a_i=\sum_i b_i=1$) splitting method \eqref{eq.splitting}, we expand its error as
$$
	S^{[m]}_p - e^{hA} = \sum_{i=p+1} \sum_{j=1} e_{i,j} \eps^j h^{i} C_{i,j},
$$
where $e_{i,j}$ is a polynomial in the splitting coefficients $a_k,b_k$ and $C_{i,j}$ is a sum of matrix products consisting of all combinations containing $(i-j)$ sparse elements $D$ and $j$ times $B$. 
Notice that in addition to the scaling $h$, we also expand in powers of the small parameter $\eps$.
The method is said to be of order $p=(p_1, p_2, \ldots)$ if $e_{i_1,1}=e_{i_2,2}=\ldots=0$ for all $i_k\leq p_k$ and $p_1\geq p_2\geq \cdots$.

Designing a method now consists of identifying the dominant error terms $e_{i,j} \eps^j h^{i}$ 
and finding coefficients $a_j, b_j$ to zero the polynomial $e_{i,j}$.
The main tool in this endeavor is the Baker-Campbell-Hausdorff formula which provides a series expansion of the single exponential that has been actually computed when multiplying two matrix exponentials,
$$
	e^{hA}e^{hB} = e^{\bch(hA,hB)}, \quad \bch(hA,hB) = h(A+B) + \frac{h^2}{2}[A,B] + \mathcal{O}(h^3).
$$
Recursive application of this formula to a symmetric splitting \eqref{eq.splitting} establishes the concept of a modified matrix $h\tilde{{A}}$, along the lines of backward-error-analysis, 
\begin{multline}\label{eq.backward}
	 \log(S^{[m]}_p) = h\tilde{{A}} =
	hA +\tilde e_{3,1}\varepsilon h^3[D,[D,B]] +\tilde e_{3,2}\varepsilon^2 h^3[B,[D,B]]\\
																	  +\tilde e_{5,1}\varepsilon h^5[D,[D,[D,[D,B]]]] +\tilde e_{5,2}\varepsilon^2 h^5[[D,[D,B]],[D,B]]\\ +\tilde e_{5,3}\varepsilon^2 h^5[B,[D,[D,[D,B]]]]																		
																		+\tilde e_{7,1}\varepsilon h^7[D,[D,[D,[D,[D,[D,B]]]]]] + \mathcal{O}\left(\varepsilon^3 h^5 + \varepsilon^2 h^7\right),
\end{multline}
where the $\tilde e_{i,j}$ are also polynomials in the splitting coefficients $a_k,b_k$ which multiply elements of the Lie algebra and are different from the coefficients $e_{i,j}$.
Higher-order terms can be computed by efficient algorithms \cite{casas09afa}.

\subsection{Error propagation by squaring}

The splitting method (\ref{eq.backward}) can also formally be written as
\begin{equation} \label{eq.Sp1p2}
	S^{[m]}_{(p_1,p_2)} = \exp\left( h(D+\varepsilon  B) +
		\varepsilon \sum_{k>p_1} c_k h^k[D^k,B]+ \mathcal{O}\left(\varepsilon^2 h^{p_2+1}\right) \right)
\end{equation}
where $[D^k,B]=[D,[D,[\ldots,[D,B]\ldots]]]$ and there is only one term proportional to $\varepsilon$ at each power of $h$. 
We can then define a \textit{processor}, a close to the identity map
\begin{equation} \label{eq.Proc1}
	P = \exp\left( -\varepsilon \sum_{k>p_1} c_k h^{k-1}[D^{k-1},B] \right),
\end{equation}
such that the method can be written as
\begin{equation} \label{eq.Sp1p2Proc}
	S^{[m]}_{(p_1,p_2)} = PKP^{-1},
\end{equation}
with
\begin{equation} \label{eq.Sp1p2K}
	K= \exp\left( h(D+\varepsilon  B) + \mathcal{O}\left(h^{p_2+1}\varepsilon^2\right) \right).
\end{equation}
Suppose now that the matrix $A$ can be diagonalized, $A=QD_AQ^{-1}$, then clearly
\begin{equation*}
	\e^{A}=Q\e^{D_A}Q^{-1}.
\end{equation*}
The \textit{kernel} $K$ of the numerical method, on the other hand, can be diagonalized for sufficiently small $h=1/n$ and $\varepsilon$ using
\[
  \hat Q = Q + \mathcal{O}\left(h^{p_2+1}\varepsilon^2\right),  \qquad
  \hat D_A = hD_A + \mathcal{O}\left(h^{p_2+1}\varepsilon^2\right),
\]
such that, after $n$ integration steps, we obtain
\begin{equation}\label{eq.Kn}
	K^n=\hat Q\e^{\tilde D_A}\hat Q^{-1}.
\end{equation}
with $  \tilde  D_A = D_A + \mathcal{O}\left(nh^{p_2+1}\varepsilon^2\right)$.
The size estimates of the above considerations lead to a favorable error propagation result which is stated in the following theorem.

\begin{theorem}\label{theorem1}
Let $A=D+\varepsilon B$ a diagonalizable matrix such that $\e^A$ is bounded and let $S^{[m]}_{(p_1,p_2)}$ be a splitting method that approximates the scaled exponential  $\e^{hA}$ with $h=1/n$. Then, 
 for sufficiently small values of $h$ and $\varepsilon$ we have that
\begin{equation}\label{eq.thm1}
	\left\|\e^A - \left( S^{[m]}_{(p_1,p_2)} \right)^n  \right\| \leq C_1 h^{p_1+1}\varepsilon +  n C_2 h^{p_2+1} \varepsilon^2.
\end{equation}
 where $C_1,C_2$ are constants which do not depend on $h$ and $\varepsilon$.
\end{theorem}

\begin{proof}
From (\ref{eq.Sp1p2Proc}) and (\ref{eq.Kn}) we have that 
\begin{equation} \label{eq.Sp1p2Proc2}
	\left(S^{[m]}_{(p_1,p_2)}\right)^n = P\hat Q\e^{\tilde D_A}\hat Q^{-1}P^{-1}
	 = \tilde Q\e^{\tilde D_A}\tilde Q^{-1}
\end{equation}
where now $\tilde Q = P\hat{Q} = Q + \mathcal{O}\left(h^{p_1+1}\varepsilon\right)$. Then
\begin{align}\label{eq.proof1}
	\left\|\e^A - \left( S^{[m]}_{(p_1,p_2)} \right)^n  \right\| & = 
	\left\|Q\e^{D_A}Q^{-1} - \tilde Q\e^{\tilde D_A}\tilde Q^{-1}  \right\|  \nonumber  \\
	& = 
	\left\|Q\e^{D_A}Q^{-1} - \tilde Q\e^{D_A}Q^{-1} + \tilde Q\e^{D_A}Q^{-1} - \tilde Q\e^{\tilde D_A}\tilde Q^{-1}  \right\| \nonumber \\
	& \leq 
	\|Q - \tilde Q\| \ \|\e^{D_A}Q^{-1}\| + \|\tilde Q\| \ \|\e^{D_A}Q^{-1} - \e^{\tilde D_A}\tilde Q^{-1} \|  \nonumber.
\end{align}
The right summand is expanded in a similar way to
\begin{align}
	\|\e^{D_A}Q^{-1} - \e^{\tilde D_A}\tilde Q^{-1} \| & = 
		\|\e^{D_A}Q^{-1} - \e^{\tilde D_A}Q^{-1} + \e^{\tilde D_A}Q^{-1} - \e^{\tilde D_A}\tilde Q^{-1} \|    \\
	& \leq 
	\|\e^{D_A} - \e^{\tilde D_A}\| \ \|Q^{-1}\| + \|\e^{\tilde D_A}\| \ \|Q^{-1} - \tilde Q^{-1} \|.  \nonumber
\end{align}
Taking into account that $\tilde D_A = D_A + \mathcal{O}\left(nh^{p_2+1}\varepsilon^2\right)$, $\tilde Q = Q + \mathcal{O}\left(h^{p_1+1}\varepsilon\right) $, and that $\e^A$ is bounded we obtained the desired result for sufficiently small values of $h$ and $\varepsilon$.
\end{proof}

This result indicates that the error is the sum of a local error of order $\mathcal{O}\left(\varepsilon\right)$ plus a global error of order $\mathcal{O}\left(\varepsilon^2\right)$.
For problems which require a relatively large number of squaring (a large value of $n=2^s$) the dominant error of the splitting methods is proportional to $\varepsilon^2$. Then, to build methods which are accurate for different values of $s$ it seems convenient to look for methods of effective order $(p_1,p_2)$ with $p_1>p_2$

The following numerical example illustrates the results obtained.

{\bf Example} Let
\beq\label{eq.example}
  A= \left( \begin{array}{cc}  \eps & 1+\eps \\ -1+\eps & - \eps
			\end{array}\right), \qquad \qquad
  D= \left( \begin{array}{cc}  0 & 1 \\ -1 & 0
			\end{array}\right)
\eeq
with $\eps=10^{-1},10^{-3}$, and approximate $\e^{2^sA}=\left(\cdots \left(\e^A\right)^2\cdots\right)^2$ to a relatively low accuracy. To approximate $\e^{A}$, we consider a fourth-order Taylor method, $T_4(A)$ (that only requires 2 products) and a fourth-order Pad\'e approximation, $r_4(A)$ (with a cost of one product and one inversion,  equivalent to $1+4/3$ products).
We compare the obtained results with the second-order splitting method \eqref{eq.leapfrog}, which we denote by $S_2^{[2,a]}$ or, since in this case $p_1=p_2=2$, $S_{(2,2)}^{[2,a]}$, where the exponential $\e^D$  is computed exactly and $\eps B$ is approximated with the second order diagonal Pad\'e method, $r_2(\eps B)$.
The exact solution is given by
\[
  \e^{2^sA}= \left( \begin{array}{cc}  
	\cos(2^s\mu) + \frac{\eps}{\mu} \sin(2^s\mu)  & \frac{1+\eps}{\mu}\sin(2^s\mu) \\
	- \frac{1-\eps}{\mu} \sin(2^s\mu) & \cos(2^s\mu) - \frac{\eps}{\mu} \sin(2^s\mu)  
			\end{array}\right)
\]
with $\mu=\sqrt{1-2\eps^2}$ and we analyze the error growth due to the squaring process in Fig.~\ref{fig.longtime}. 
We observe that neither Pad\'e nor Taylor methods are sensitive w.r.t. the small parameter, whereas the splitting method drastically improves when decreasing $\varepsilon$.
The splitting method is only of second order and thus used with a second order Pad\'e method $r_2$ (using the fourth order method $r_4$ leaves error plot unchanged).
Notice that for the small perturbation $\varepsilon=10^{-3}$, the splitting with $r_2(\eps B)$ is more accurate than the fourth-order Pad\'e $r_4(A)$ which comes at nearly twice the computational cost (1 inversion vs. 1 inversion and 1 dense product). According to Theorem~\ref{theorem1}, the error of $S_{(2,2)}^{[2,a]}$ is the sum of a local error proportional to $h^{3}\varepsilon$ and a global error proportional to  $n  h^{3} \varepsilon^2$, with $n=2^s$. Fig.~\ref{fig.longtime} shows the results obtained for different values of $\varepsilon^2$ and $s$ which clearly show both error sources.


\begin{figure}[!ht]
\centering
	\pgfplotsset{every axis plot/.append style={line width=1.2pt, mark size=2pt},
		tick label style={font=\footnotesize},
		every axis/.append style={%
		minor x tick num=2,
		minor y tick num=2,
		minor z tick num = 2,
		scale only axis, 
		font=\footnotesize,
		}
	}
	\setlength\figurewidth{.4\textwidth}%
	\setlength\figureheight{.2\textwidth}
	\tikzsetnextfilename{figs/longtime_final}
	\includegraphics{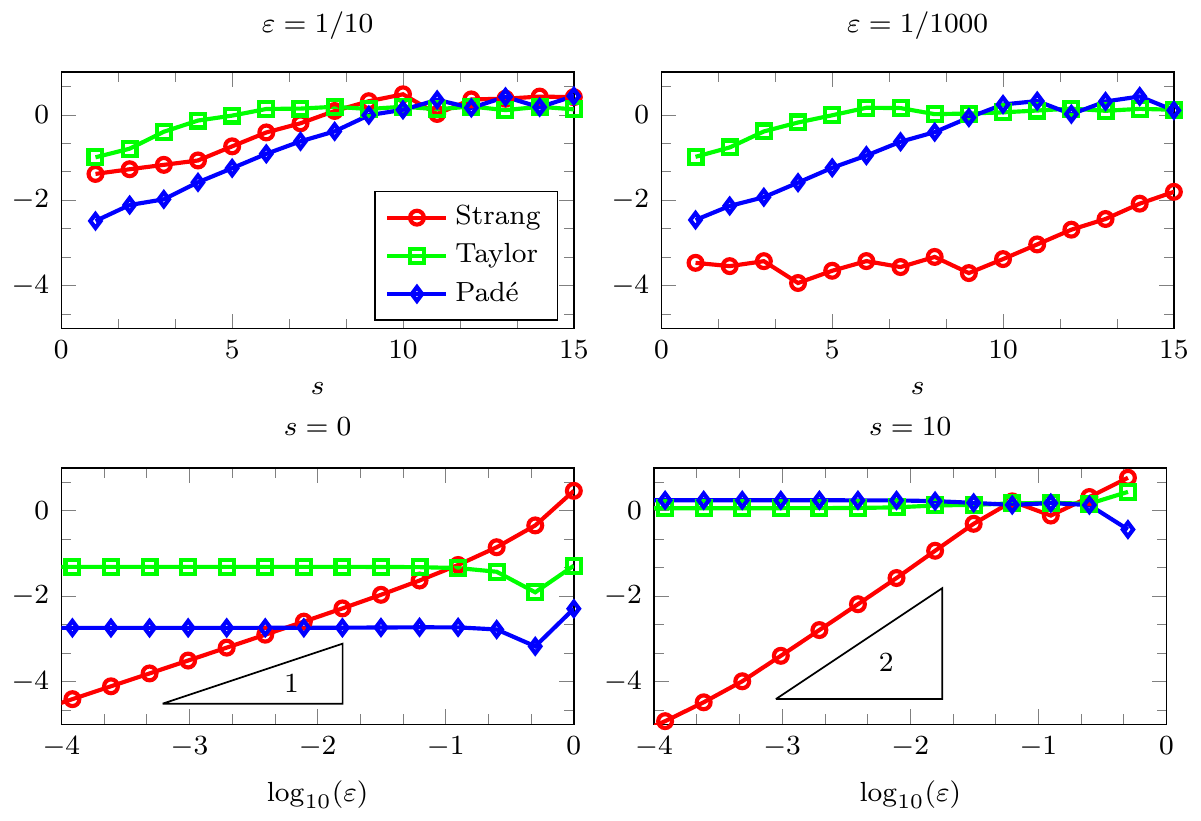}

\caption{\label{fig.longtime} Error in the approximation to $\e^{2^sA}$ with $A$ given by \eqref{eq.example} for different values of $\varepsilon$ and number of squaring, $s$, in double-logarithmic axes. The bottom figures show that the error of the splitting methods is proportional to $\varepsilon$ for small $s$ (local error) and proportional to $\varepsilon^2$ for large values of $s$ (global error)}
\end{figure}


\section{Splitting methods for scaling and squaring}\label{sec:splitting}
Taking into account the numerical effort established in the introduction, we derive methods which are optimal for the problem at hand. 
The optimization principle becomes clear at the example of the two versions of Strang's second-order splitting method
\def\e{e}
\begin{align}\label{eq:StrangABA}
  & &S_2^{[2,a]}&=
  \e^{\frac{h}2 D} \ \e^{h \varepsilon B} \ \e^{\frac{h}2 D}
  =\mathcal{D}_{h/2}\mathcal{B}_{h}\mathcal{D}_{h/2},\\
	\label{eq:StrangBAB}
\text{and} & &
  S_2^{[2,b]}&=
  \e^{\frac{h}2 \varepsilon B} \ \e^{h D} \ \e^{\frac{h}2 \varepsilon B}
    =\mathcal{B}_{h/2}\mathcal{D}_{h}\mathcal{B}_{h/2},
\end{align}
which differ in computational cost: 
Using the notation $\mathcal{D}_h=e^{hD}$, $\mathcal{B}_h=e^{hB}$, and keeping in mind that $\mathcal{D}_{h}$ is a sparse matrix while $\mathcal{B}_{h}$ is dense, the dominant numerical cost amounts to a single exponential with $\cost(S_2^{[2,a]})=\cost(\mathcal{B}_{h})$ for the first version, whereas the latter requires an additional matrix product, $\cost(S_2^{[2,b]})=\cost(\mathcal{B}_{h/2})+\cost(\mathcal{B} \mathcal{B})$.

Furthermore, the large dominant part $D$ is multiplied by $1/2$ before exponentiation in the cheaper variant which is advantageous in the sense of the scaling process.

We follow a variety of strategies in order to develop new methods and group them according to the splitting terminology, keeping in mind that the costly parts are products and exponentials of the dense matrices $\mathcal{B}$ and $B$, respectively.

\subsection{Standard splittings}
As we have discussed for the Strang splitting $S_2^{[2,b]}$, despite the appearance of $B$ in two exponents, only one exponential actually has to be computed which is then stored and reused for the second identical exponent.

Generalizing this principle, we search for splitting methods $a_i,b_j$ where all $b_j=b$ are identical to reduce the computational effort which now comes solely from the dense-matrix multiplications.
A composition that is also symmetric in the coefficients $a_j$ will reduce a great number of error terms (since even powers in $h$ disappear) and additionally the amount of (cheap) exponentials $\cD$ to be computed.

Next, we derive a particular family of splittings which can be understood in analogy to squarings and allows to reduce the necessary products.

\subsubsection{Modified squarings}
We propose to replace a given number of squarings by a one-step splitting method which has the benefit of free parameters to minimize the error.
For illustration, let us compute a squaring step, $h=2^{-1}$, of the standard Strang method,

\begin{equation} \label{eq.2Strangs}
	(e^{h/2A}e^{hB}e^{h/2A})^2 = e^{\frac14 A}e^{\frac12 B}e^{\frac12 A}e^{\frac12 B}e^{\frac14 A},
\end{equation}
which we then contrast with a general splitting method at the same cost (one exponential and one product) without squaring ($h=1$),

\begin{equation} \label{eq.2Steps}
	e^{a_2 A}e^{\frac12 B}e^{a_1 A}e^{\frac12 B}e^{a_2 A}.	
\end{equation}
It is evident that (\ref{eq.2Steps}) includes (\ref{eq.2Strangs})  as a special case (choosing $a_1=1/2, a_2=1/4$) and we use the example \eqref{eq.example} to illustrate the gains in accuracy. Fig.~\ref{fig.modified} shows that the performance is very sensitive to the choice of the free parameter and the method of effective order $(4,2)$ is very close to the optimal one.
\begin{figure}[!ht]
\begin{center}
	\pgfplotsset{every axis plot/.append style={line width=1.2pt, mark size=2pt},
		tick label style={font=\footnotesize},
		every axis/.append style={%
		minor x tick num=2,
		minor y tick num=2,
		minor z tick num = 2,
		scale only axis, 
		font=\footnotesize,
		}
	}
	\setlength\figurewidth{.4\textwidth}%
	\setlength\figureheight{.4\textwidth}
	\tikzsetnextfilename{figs/squaring}
	\includegraphics{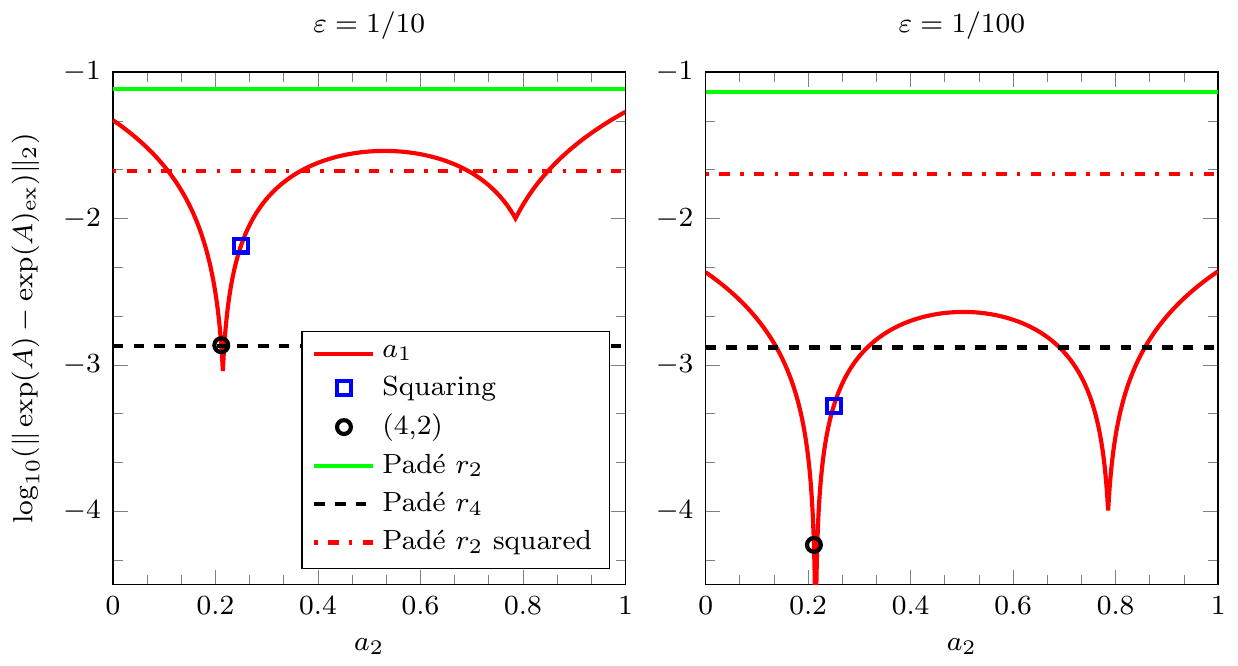}
\end{center}
\caption{\label{fig.modified}Modified squarings. All methods apart from $r_2(A)$ (green solid) have approximately the same numerical cost since the split uses 2nd order pad\'e}
\end{figure}
A larger number of squarings $s$ can be replaced by a recursive procedure,
$$
	X_0 = e^{h b \eps B}, \qquad
	X_k = X_{k-1}e^{a_k h D}X_{k-1}, \qquad k=1,\ldots, s
$$
and $Y_{s} = e^{a_{s+1}h}X_{s}e^{a_{s+1}h}$ where $b=1/2^s$.
The costly multiplications occur in the consecutive steps, $X_k$, where we recycle already computed blocks while introducing free parameters $a_k$ at negligible extra effort. As a result, the cost of the algorithm is
\[
  c(Y_{s}) = s + c(e^{h b \eps B})
\]
where it usually suffices to approximate $e^{h b \eps B}$ with a second or fourth-order Pad\'e method, so $c(e^{h b \eps B},r_2)=\frac43$ and $c(e^{h b \eps B},r_4)=1+\frac43$.
For consistency, the coefficients $a_k$ have to satisfy
$$
 \left( 2^{s-1} a_1  + \cdots+ 2 a_{s-1} +	a_{s} \right)+ 2a_{s+1} = \sum_{k=1}^{s} 2^{s-k} a_k + 2 a_{s+1}= 1.
$$
Notice that the choice $a_{s+1}=1/2^{s+1}$, $a_k=1/2^s$ for $k=1,\ldots, {s}$, corresponds to the standard scaling and squaring applied to the Strang method \eqref{eq:StrangABA}.
In the following, we have collected the most efficient splitting methods for an increasing numbers of products $s=0,1,2,3,4$.
We have observed in the numerical experiments that for $s>4$, the gain w.r.t. to standard scaling and squaring is marginal, and they are not considered in this work.

However, the parameter $h$ demonstrates how any such method can be combined with standard scaling and squaring. 

This procedure is equivalent to consider the partition $s=s_1+s_2$ where the first $s_1$ squarings are carried out with the recursive algorithm with $b=1/2^{s_1}$ and we continue with the remaining standard $s_2$ squarings with  $h=1/2^{s_2}$.

\paragraph{$s_1=0$}
Strang $S_2^{[2,a]}$ with
local order $\cO(\eps h^3)$.

\paragraph{$s_1=1$}
After imposing symmetry, one free parameter remains and is used to obtain (4,2) methods \cite{laskar01hos,mclachlan95cmi},
\begin{equation}\label{eq:42ABA}
  Y_1	=
  \mathcal{D}_{ha_2}\mathcal{B}_{h/2}\mathcal{D}_{ha_1} \mathcal{B}_{h/2}\mathcal{D}_{ha_2},
\end{equation}
where $a_2=(3-\sqrt{3})/6, \ a_1=1-2a_2$ and  with
local order $\cO(\eps h^5+\eps^2 h^3)$.

\paragraph{$s_1=2$}
Allowing an additional product, at $b=1/4$, we have
\beq\label{eq.five}
		Y_2 = \cD_{a_3h} (\cB_{h/4} \cD_{a_2h} \cB_{h/4}) \cD_{a_1h} (\cB_{h/4} \cD_{a_2h} \cB_{h/4}) \cD_{a_3h}.
\eeq
Optimizing the free parameters $a_3, a_2$, (where for consistency $a_1=1-2(a_3+a_2)$) we can construct fourth-order methods, although complex-valued, with $a_3 = \frac1{10}(1-i/3), a_2 = \frac{2}{15}(2+i)$ and their complex conjugates $a_i^*$ \cite{castella09smw}.
Alternatively, there are six real-valued (6,2) methods, the best of which is given in Table~\ref{tab.coefs}.

\paragraph{$s_1=3$}
The three parameters for $Y_3$ can be used to produce complex-valued methods of order (6,4) or real-valued methods of order (8,2), the ones with smallest error coefficients can be found in Table~\ref{tab.coefs}.

\paragraph{$s_1=4$}
The next iteration yields a 17-stage method $Y_4$.
Its four parameters can be used to cancel the error coefficients $e_{3,1}, e_{3,2}, e_{5,1}, e_{7,1}$
for 48 complex (8,4) methods, or a $(10,2)$ method with positive real coefficients, see Table~\ref{tab.coefs}.

\begin{table}[!ht]
\caption{\label{tab.coefs}Modified squarings with and without commutators. In the right column, the corresponding computational cost is given together with the number of omitted solutions of the order conditions.}
\noindent{\footnotesize
\begin{tabular}{ll}
\toprule
$Y_2$, order (6,2)
 & $\cost(\cB_{h/4}) + 2\cost(\cB \cB)$\\[1mm]
$a_1 = \sqrt{(5-\sqrt{5})/30}, \ 
 a_2 = \sqrt{ (5-2\sqrt{5})/15} $
& 
	[7 solutions omitted]\\
\midrule
$Y_3$, order (8,2)\\[2mm]
	$a_1= 0.153942020841153420134790213164$ & only positive solution\\
	$a_2= 0.089999237645462605679630986655$ & [47 omitted]\\
	$a_3= 0.102244554291437558627161030779$ \\
	$a_4=\frac12- (4 a_1 + 2a_2 + a_3)/2$.\\[1mm]
	 					
$Y_3$, order (6,4)		& $\cost(\cB_{h/8}) + 3\cost(\cB \cB)$\\[2mm]
	$a_1 = 0.13534452760420860194  + 0.06201309787740406230i$ & [7 omitted]\\ 
	$a_2 = 0.13027125534284511606  - 0.10310039626441585374i$\\
	$a_3 = 0.099062332740825337251 - 0.015885424766237390724i$\\ 
	$a_4=\frac12- (4 a_1 + 2a_2 + a_3)$\\[1mm]

\midrule
$Y_4$, order (10,2)&  $\cost(\cB_{h/16}) + 4\cost(\cB \cB) $\\[2mm]
$a_1 = 0.077255933048297137202077893145$ & only positive solution\\ 
$a_2 = 0.0444926322393204245189059370354$& [383 omitted]\\ 
$a_3 = 0.051080773613693429438027986467$ \\ 
$a_5 = 0.0254553659841308990458390646508$\\ 
$a_4 = 1  - 8a_1 - 4a_2 - 2a_3 - 2a_5$\\[1mm]
							
$Y_4$, order (8,4)\\[1mm]
$a_1 = 0.06782965853562196485274129 + 0.03038453954138687801299186i$  &  [47 omitted]\\
$a_2 = 0.06477414774829711915884478 - 0.05170904068177844632921239i$    & 		\\
$a_3 = 0.04963134399080347125041612 + 0.00584283681423207753349501i$ \\ 		  	
$a_5 = 0.02474856149827627051056177 - 0.00610084851840072905292033i$		\\
$a_4 = 1  - 8a_1 - 4a_2 - 2a_3 - 2a_5$\\
\midrule
\midrule
$\tilde{Y}_2$, order (6,4), minimizing $\cO(\varepsilon^2 h^5)$\\[2mm]
$a_1=(1-a_2-2a_3)/2$\\
$a_2=0.47071989362081947165$\\
$a_3 = 0.04898669326146179875$\\
$\beta = -0.002320917859694561351$\\
$\gamma=0.0000329546718228203782$\\[1mm]
$\tilde{Y}_2$, order (8,4) &[47 omitted] \\[1mm]
$a_1=0.3602258146389491220734647$ \\ 
$a_2=1 - 2(a_3 + a_1)$\\
$a_3=0.0766102130069293861483005$ \\
$\beta = -0.00103637077918270398691258$\\
$\gamma = 0.000010240482532598594411391$\\
%
\bottomrule
\end{tabular}
}
\end{table}

\subsection{Modified splittings}
A drastic improvement on the previous methods can be made through the use of commutators.
The special structure of the matrix allows for the fast computation of certain commutators, namely the ones that contain the matrix $B$ only once. 
The inclusion of these commutators in the scheme will not only allow to reduce the number of error terms but also to reach order 4 using only real coefficients.
Since we are interested in symmetric methods of up to order (6,4), the relevant terms are
\begin{align*}
	[D, [D, B]] &= DDB - 2DBD + BDD, \\ 
	[D^4, B] 
	&= DDDDB - 4DDDBD + 6DDBDD - 4 DBDDD + BDDDD,
\end{align*}
and neglecting the numerical cost of summation and multiplication by a sparse matrix $D$, it is clear that the exponential
\[
	e^{\alpha h B + \beta h^3 [D,[D,B]] + \gamma h^5 [D, [D, [D, [D, B]]]]} = \tilde{\cB}_{\alpha,\beta,\gamma}
\]
can be evaluated at the same cost as $\cB_{\alpha h}$.
Along the lines of the modified squarings, we have derived the following compositions which require only one exponentials $\tilde{B}$ at a fixed number of products. 
The substitution $Y_s\to\tilde{Y}_s$ indicates the replacement of $B$ by $\tilde{B}$.
\paragraph{$s=0$} Strang's method can be made into a (6,2) scheme with
\beq\label{eq:strang62}
	\tilde{Y}_0 = \cD_{h/2} \tilde{\cB}_{1, 1/24, 1/1920}\cD_{h/2}.
\eeq
We stress that, in principle, a method of order $(2n,2)$ can be constructed using only a single exponential, however, at the expense of increasingly complicated commutators, $[D,[D,[\ldots,[D,B]]\cdots]$ whose computational complexity cannot be neglected anymore.
\paragraph{$s=1$} 
Replacing $\cB_{h/2}$ by $\tilde{\cB}$ in \eqref{eq:42ABA}, we obtain the (6,4) method
\beq
		\tilde{Y}_1 = \cD_{ha_2}\tilde{\cB}\cD_{ha_1} \tilde{\cB}\cD_{ha_2},
\eeq
where $a_2=1/6, \ a_1=2/3$ and $\tilde{\cB}_{1/2, -1/144, 121/311040}$ with unchanged effort $\cost(\cB_{h/2})+\cost(\cB \cB)$.
\paragraph{$s=2$} 
Using one additional multiplication, we reach $\tilde{Y}_2$,
which can be tuned to be of order (8,4) or (6,4) while minimizing the error at $\mathcal{O}(\varepsilon^2 h^5)$, see Table~\ref{tab.coefs}.

We have also analyzed other classes of splitting and composition methods. The methods obtained showed  a worst performance on the numerical examples tested in this work. The schemes obtained are, however, collected in the appendix for completeness.

\section{Error analysis}
Our methods have proven successful for a low to medium accuracy since the high-order Pad\'e methods are hard to beat at round-off precision.
In a first step, we derive new scaling estimates for Pad\'e methods for lower precision requirements following \cite{higham09tsa}.
Let $\theta_m(u)$ be the largest value of $\|A\|$ s.t. the Pad\'e scheme $r_{2m}$ has precision at least $u$, i.e.,
\[
	\forall A, \|A\|\leq \theta_m \ :\  r_{2m}(A)=e^{A+E}, \ \text{s.t.}\ \|E\|\leq u.
\]	
The new $\theta_m$ are given in Table~\ref{tab.theta}.
\begin{table}\centering\footnotesize
\caption{\label{tab.theta}Theta values for diagonal Pad\'e of order $2m$ with minimum number of products. 
The numbers highlighted in boldface correspond to the minimal cost $\pi_{2m}-\log_2(\theta_{2m})$ }
\newcolumntype{H}{@{}>{\lrbox0}l<{\endlrbox}}
\newcolumntype{D}{>{$}r<{$}}
\begin{tabular}{*{8}{D}*{5}{H}D}
\toprule
	u\backslash m	&1				&  			2&		  3 &  		  4&		   5&  		  6&       7&   8&   9&  10&  11&  12&  13\\ \midrule
\leq2^{-53}&
				3.65\en8&5.32\en4&1.50\en2&8.54\en2&2.54\en1&5.41\en1&9.50\en1&1.47&2.10&2.81&3.60&4.46&\bf 5.37\\
\leq1\en10 & 
				3.46\en5&1.64\en2&1.47\en1&4.73\en1&9.98\en1&1.69		 &\bf 2.51&3.44&4.44&5.51&6.62&7.76&8.94\\
\leq1\en6& 
				3.46\en3&1.64\en1&6.80\en1&1.49		 &\bf 2.48&3.58		 &4.76		&5.98&7.24&8.52&9.81&1.11\ep1&1.24\ep1\\
\bottomrule

\end{tabular}
\end{table}
It is clear that the number of necessary scalings for a sought precision is $s=\lceil{\log_2(\|A\|/\theta_m)\rceil}\in\mathbb{N}_0$ and taking into account the number of multiplications $\pi_m$ needed with each method, a global minimum $s+\pi_m$ can be found at each precision.

We will focus our attention on the medium precision range $u\leq 10^{-6}$, where the 10th order method $r_{10}$ is optimal among the Pad\'e schemes.
In analogy to the error control for Pad\'e methods, we discuss the backward error of the previously obtained splitting methods.
The BCH formula, in the form \eqref{eq.backward}, already gives us a series expansion of the remainder $E$,
\beq\label{eq.error}
E = \sum_{i=p+1}\sum_{j=1} h^i f_{i,j} \mathbf{C}_{i,j}.
\eeq
However, the expansion is difficult to compute for $i>15$ with exponentially growing effort in the symbolic computation.
Further complications arise from the nature of the expansion: it involves commutators $\mathbf{C}_{i,j}$ in $D,B$ which we have to estimate.
For most cases, the roughest (although sharp) estimate
\beq\label{eq.rough}
	\|[D,B]\|=\|DB-BD\|\leq 2\eps\|D\|^2, \qquad \eps = \|B\|/\|D\|,
\eeq
is way to loose to give accurate results.
Having in mind matrices with asymmetric spectra, i.e., small positive and large negative eigenvalues,
the following estimate is more useful \cite[Theorem 4]{kittaneh09nif},
\[
	\|[D,B]\|\leq \|B\| (d^+-d^-),
\] 
where the numerical range of $D$ (or easier: the eigenvalues) lies within $[d^-,d^+]$, which corresponds to a factor 2 gain in the estimate.
In any case, we can refine the estimate by recycling the calculations for the modified splittings, $[D,[D,B]]$, $[D,[D,[D,[D,B]]]]$ and intermediate steps, $[D,B]$, etc.
Then, we estimate the most relevant commutators, recalling the notation $[D^2,B]=[D,[D,B]]$,
\begin{align*}
	\|[B,[D,B]]\|&\leq 2\|[D,B]\| \|B\|, \\
	\|[B,[D,[D,[D,B]]]]\|&\leq 
	2 \| [D,B]\|\, \|[D,[D,B]]\| ,
		\\
	\|[D,[B,[D,[D,B]]]]\|&
	\leq 2 \| [D,B]\|\, \|[D,[D,B]]\|
	,\\
	\|[B,[B,[D,[D,B]]]]\|&\leq 4 \| B\|^2 \ \|[D,[D,B]]\|,\\
	\|[D,[D,[D,[D,[D,[D,B]]]]]]\|&\leq (d^+-d^-)^2\|[D,[D,[D,[D,B]]]]\|.
\end{align*}
The splitting methods studied in this work can be classified by their order and the leading error commutators are collected in Table~\ref{tab.split.commutators}. 

In principle, one could use the error terms at the next larger power in $h$ to estimate the quality of this truncation, but for practical purposes and $h\ll 1$, numerical experiments show that the simpler bounds are sufficient to get a reasonable recommendation for the number of squarings.
For illustration, we print the expansion \eqref{eq.error} for the method 
\eqref{eq:strang62} 
\begin{align}\label{eq.err62}
	E^{[6,2]}(h)\leq \tilde{E}^{[6,2]}=&\; 3.11\en{6} h^7 \|[D^6,B] \|
							 + 8.33\en{2} h^3 \|[B,[D,B]]\| \\ \nonumber
							&+  h^5 (1.39\en{3} \|[B,[D^3,B]]\| + 5.56\en{3} \|[[B,D],[D^2,B]]\|)\\ \nonumber
	& +  h^5 (5.56\en{3} \|[B^2,[D^2,B]]\| + 2.78\en{3} \|[[B,D],[B^2,D]]\|\\ \nonumber
	&+\mathcal{O}\left(\eps h^9 + \eps^2 h^7 + \eps^3 h^7\right)
\intertext{and for method $\tilde{Y}_2$ of order (6,4) from Table~\ref{tab.coefs},}
\label{eq.err64}
	E^{[6,4]}(h)\leq \tilde{E}^{[6,4]}=&\; 3.49\en{5} h^7 \|[D^6,B] \| 	\\ \nonumber 
	&	+  h^5 (1.70\en{3} \|[B,[D^3,B]]\| + 1.39\en{3} \|[[B,D],[D^2,B]]\|)\\ \nonumber 
	& +  h^5 (1.39\en{3} \|[B^2,[D^2,B]]\| + 4.63\en{4} \|[[B,D],[B^2,D]]\|\\ \nonumber 
	&+\mathcal{O}\left(\eps h^9 + \eps^2 h^7 + \eps^3 h^7\right) .
\end{align}
Then, the following algorithm suggests itself: 
Compute the commutators needed for the modified squarings, estimate their norms and finally evaluate the polynomials $\tilde{E}(h)$ to find an upper bound for $h$ such that the local error remains below given accuracy $u$. This $h$ translates directly to the number of external squarings $s_2=\lceil\log_2(h)\rceil$ and now, it only remains to sum the computational cost originating from the number of dense products and exponentials to find the overall most efficient method for a particular set of matrices $D,B$.
In contrast to the static Pad\'e case, where there is a single best method by just fixing the precision, this procedure is more flexible and chooses - at virtually no extra cost - the best method for the given matrix algebra structure.

Furthermore, we can establish a threshold for the size of the small parameter $\varepsilon$ in order to decide when splittings should be preferred over Pad\'e methods.
For example, let $u=10^{-6} (10^{-4})$ be the desired precision, we then know that $r_{10}$ ($r_{10}$) is optimal and the largest value the norm $\theta =\|A\|$ can take is $\theta_5=2.48 (\theta_5=3.85)$. 
Given that $r_{10}$ requires three multiplications, we use the splitting method $\tilde{Y}_0$ with three squarings to yield a method of the same computational cost. In \eqref{eq.err62}, this corresponds to taking $h=2^{-3}$.
Applying the roughest possible estimate \eqref{eq.rough} to $\tilde{E}^{[6,2]}(2^{-3})$, we obtain a polynomial in $\eps$ which takes values below $u$ for $\eps\leq 0.01 (0.05)$.
In practice, the norm estimates are sharper since we can use the commutators that have been computed in the algorithm and we expect an even larger threshold for $\eps$.

\begin{table}[!tbh]\caption{\label{tab.split.commutators}Leading error commutators at given order.
}
\begin{tabular}{l*{3}{>{$}l<{$}}}
\toprule
order & \eps^1 & \eps^2 & \eps^3\\ \midrule
$(2n,2)$ & [D^{2n},B] & [B,[D,B]] & [B,[B,[D,[D,B]]]]\\
$(2n,4)$ & [D^{2n},B] & [B,[D,[D,[D,B]]]] , [D,[B,[D,[D,B]]]] & [B,[B,[D,[D,B]]]]\\
\bottomrule
\end{tabular}
\end{table}

\section{Numerical results}
In a couple of test scenarios, we attempt to provide an idea about when our new methods are superior to standard Pad\'e methods. 
In each setting, we define a different matrix $D$ which will be perturbed by a matrix $B$, s.t. 
\[
	B_{i,j}=k(i-j)/(i+j)
\]
and $k$ is chosen to satisfy $\varepsilon=\|B\|_1/\|D\|_1$ for the parameter set $\varepsilon=10^{-1}, 10^{-2}, 10^{-3}$. 
We measure the relative error in the 1-norm, $\| S_p^{[m]} - e^A\|_1/\|e^A\|_1$ for all methods where the exact solution is computed by a high-order Pad\'e method and all splittings use the second-order scheme $r_2$ to approximate the exponential $\exp({2^{-s}B})$.

\subsection{Rotations}
Letting 
\[
	D=i \diag\{-25,-24.5,\ldots ,24.5,25\}
\]
with $i=\sqrt{-1}$, the performance of Pad\'e methods of order 10 and 26, together with the 16th-order Taylor method using 6 products is studied. Fig.~\ref{fig.case1} shows the relative error (in logarithmic scale) versus cost (number of matrix--matrix multiplications) for different choices of the scaling parameter, $s$. The horizontal line shows the tolerance desired for the numerical experiments. 
It is evident that, as expected, the Pad\'e method $r_{10}$ is the most efficient among these standard schemes and will be used for reference in later experiments.
For illustration, Fig.~\ref{fig.case1} also includes two modified squaring methods without commutators ($Y_2$, order (6,2) and $Y_3$, order (6,4) from Table~\ref{tab.coefs}), both of which are more efficient than $r_{10}$ in the lower precision range. Notice that, since $A$ isa complex matrix, to use splitting methods with complex coefficients does not increase the cost of the algorithms in this case.
Furthermore, the standard methods are insensitive w.r.t. the small parameter $\varepsilon$, whereas the splitting methods improve as $\varepsilon$ decreases.
\begin{figure}[!ht]
\centering
	\pgfplotsset{every axis plot/.append style={line width=1.2pt, mark size=2pt},
		tick label style={font=\footnotesize},
		every axis/.append style={%
		minor x tick num=1,
		minor y tick num=4,
		minor z tick num = 2,
		scale only axis, 
		font=\footnotesize
		}
	}
	\setlength\figurewidth{.24\textwidth}%
	\setlength\figureheight{.4\textwidth}
	\tikzsetnextfilename{figs/case1a}
	\includegraphics{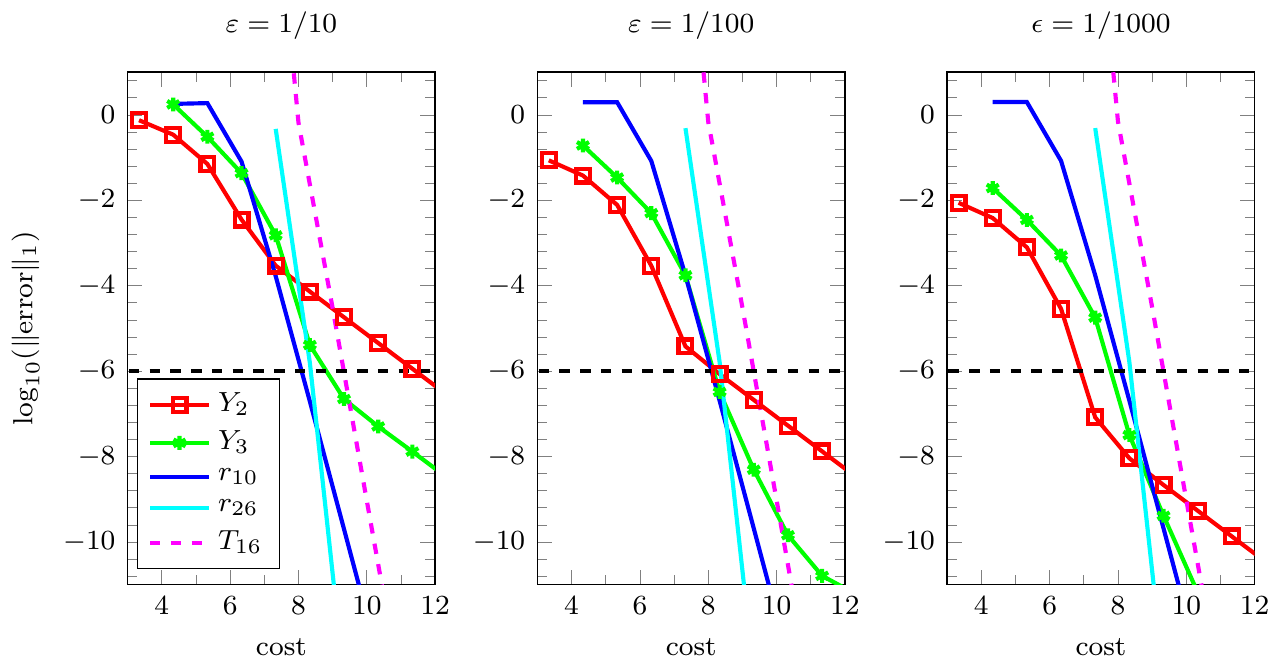}
	\caption{\label{fig.case1} Relative error  (in logarithmic scale)  versus computational cost given by the number of dense matrix-matrix products for the standard Pad\'e and Taylor methods $r_{10}, r_{26}, T_{16}$, and the splitting methods $Y_2$ and $Y_3$ of order (6,2) and (6,4), respectively, without commutators from Table~\ref{tab.coefs}. }
\end{figure}
In a second experiment in Fig.~\ref{fig.case1b}, we use the same matrices as before but choose the most efficient splitting methods with commutators, $\tilde{Y}_0$ and $\tilde{Y}_1$. 
Using the local error estimates in \eqref{eq.err62} and \eqref{eq.err64}, we indicate the point which corresponds to the optimal number of squarings for the splitting methods and compare it with the recommended squaring parameter for Pad\'e $r_{10}$. 
For a relatively large parameter $\varepsilon$ in the left panel of Fig.~\ref{fig.case1b}, the method $r_{10}$ is still superior but is already equaled in terms of computational cost for a smaller perturbation in the center plot, but at higher accuracy. As $\varepsilon$ becomes smaller in the right panel, we achieve higher accuracy at lower computational cost, saving one product for $\tilde{Y}_1$ and two products for $\tilde{Y}_2$, respectively.

\begin{figure}[!ht]
\centering
	\pgfplotsset{every axis plot/.append style={line width=1.2pt, mark size=2pt},
		tick label style={font=\footnotesize},
		every axis/.append style={%
		minor x tick num=1,
		minor y tick num=4,
		minor z tick num = 2,
		scale only axis, 
		font=\footnotesize
		}
	}
	\setlength\figurewidth{.24\textwidth}%
	\setlength\figureheight{.4\textwidth}
	\tikzsetnextfilename{figs/case1b}
	\includegraphics{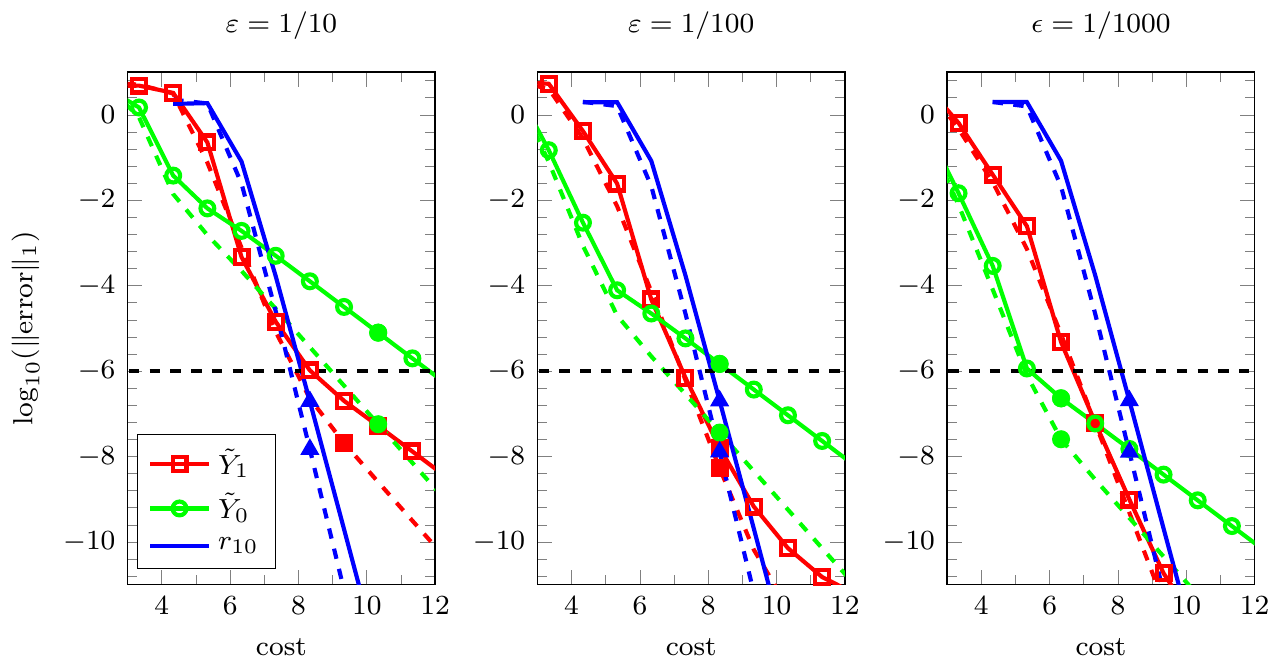}
	\caption{\label{fig.case1b} The solid lines show the relative global error $e^{A}$ after squaring versus the overall computational cost and the dashed curves depict the relative local error in $e^{2^{-s}A}$ (before squaring) which is used for the error estimate, both for Pad\'e and the splittings. 
	The filled markers indicate the position of the recommended (automatic) algorithm.
}
\end{figure}

In the next plot, Fig.~\ref{fig.case1c}, we increase the norm of the matrix and set $D_2=100 D$, and $B$ is scaled accordingly to maintain the quotient $\|B\|_1/\|D_2\|_1=\varepsilon$.
The implications are a substantial increase in the number of necessary squarings with prior scaling and corresponds to a long-time integration in which we observe the favorable behavior expected from Fig.~\ref{fig.longtime}.
The gain with respect to Pad\'e's method is striking as $\varepsilon$ decreases.

\begin{figure}[!ht]
\centering
	\pgfplotsset{every axis plot/.append style={line width=1.2pt, mark size=2pt},
		tick label style={font=\footnotesize},
		every axis/.append style={%
		minor x tick num=1,
		minor y tick num=4,
		minor z tick num = 2,
		scale only axis, 
		font=\footnotesize
		}
	}
	\setlength\figurewidth{.24\textwidth}%
	\setlength\figureheight{.4\textwidth}
	\tikzsetnextfilename{figs/case1c}
	\includegraphics{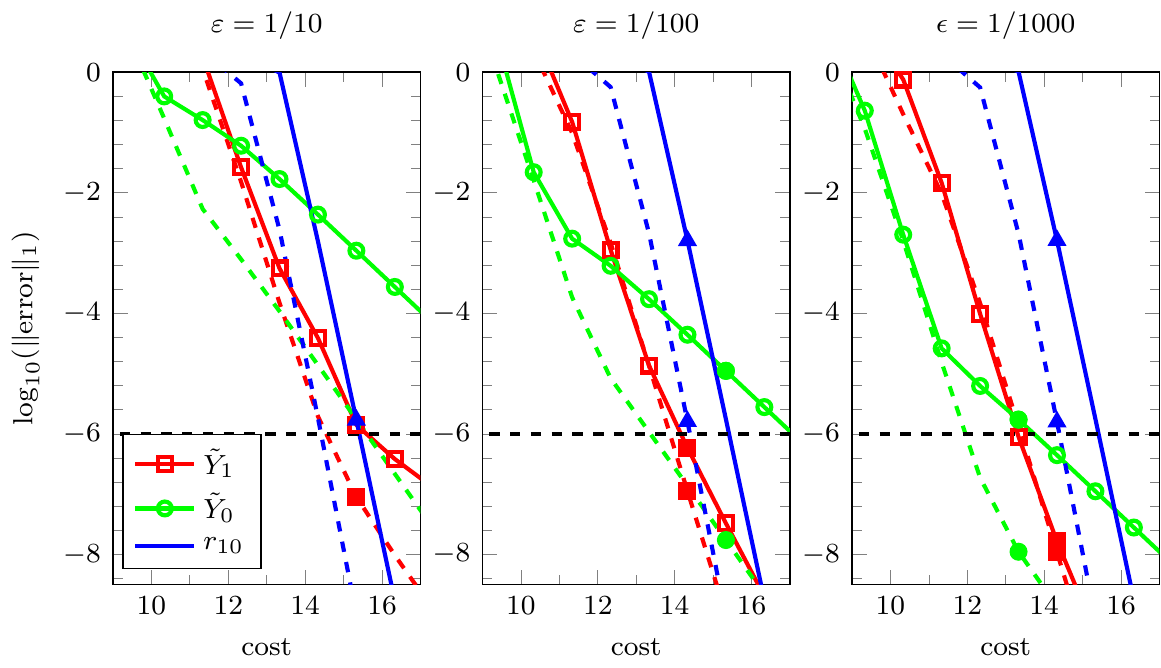}
	\caption{\label{fig.case1c} Same as Fig.~\ref{fig.case1b} for an exponential of a large norm matrix, with diagonal part $D_2=100D$.
}
\end{figure}
	
\subsection{Dissipation}
A less favorable problem for our algorithm is given using a stiff matrix with large positive and negative eigenvalues, 
\[
D=\diag\{15,14.5,\ldots, -14.5,-15\}.
\]
The perturbation $B$ is scaled as before to $\|B\|/\|D\|=\varepsilon$. Fig.~\ref{fig.case1} shos the results obtained.
Again, our methods perform well for low accuracies for not too large perturbations and improve as $\varepsilon$ becomes smaller.
\begin{figure}[!ht]
\centering
	\pgfplotsset{every axis plot/.append style={line width=1.2pt, mark size=2pt},
		tick label style={font=\footnotesize},
		every axis/.append style={%
		minor x tick num=1,
		minor y tick num=4,
		minor z tick num = 2,
		scale only axis, 
		font=\footnotesize
		}
	}
	\setlength\figurewidth{.24\textwidth}%
	\setlength\figureheight{.4\textwidth}
	\tikzsetnextfilename{figs/case2}
	\includegraphics{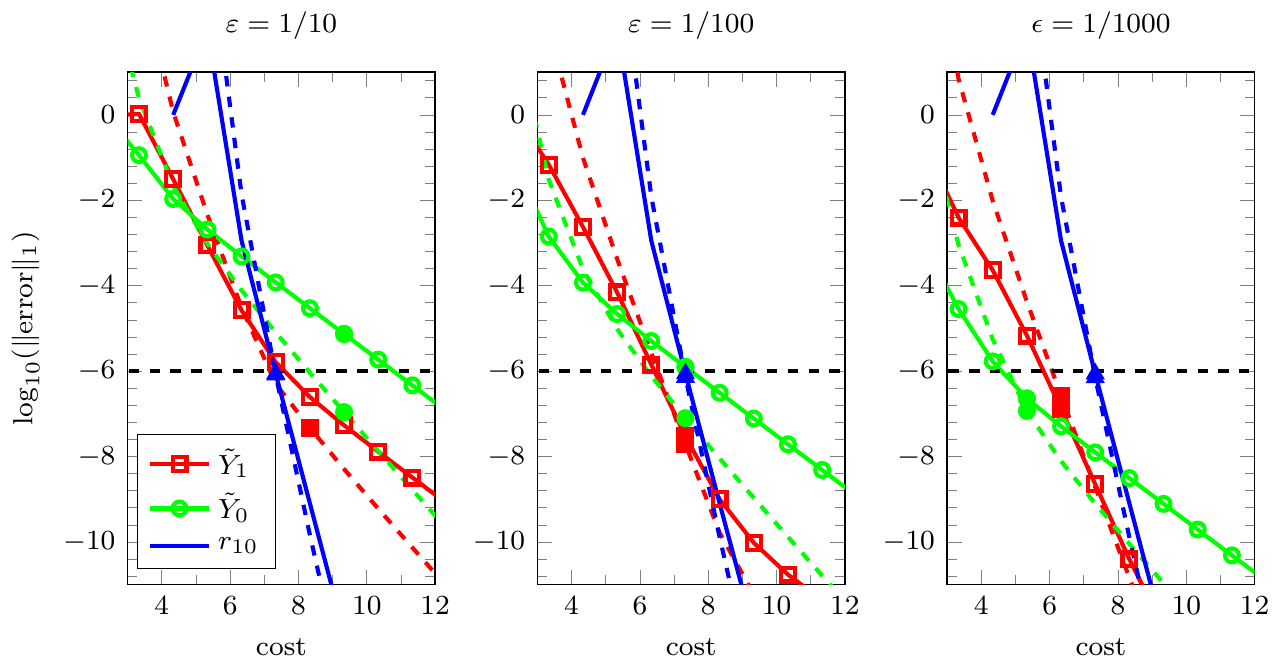}
\caption{\label{fig.case2}  Same as Fig.~\ref{fig.case1b} but for the stiff matrix case $D=\diag\{15,14.5,\ldots, -14.5,-15\}$.  }
\end{figure}

\section{Conclusions}

We have proposed a new recursive algorithm based on splitting methods for the computation of the exponential of perturbed matrices which can be written as the sum $A=D+\varepsilon B$ of a sparse and efficiently exponentiable matrix $D$ with sparse exponential $e^D$ and a dense matrix $\varepsilon B$ which is of small norm in comparison with $D$. 
We have considered the scaling and squaring technique but replacing the  Pad\'e or Taylor methods to compute the exponential of the scaled matrix by an appropriate splitting methods tailored for this class of matrices. We have proposed a recursive algorithm which allows to save computational cost and still leaves some free parameters for optimization. An important feature of splitting methods for perturbed problems is that the error is a sum of a local error of order ${\cal O}(\varepsilon)$ plus a global error of order  ${\cal O}(\varepsilon^2)$ and this allows to build new methods with high performance when low to medium accuracy is desired. 
The new schemes are built taking into account that the dominant computational cost arises from the computation of dense matrix products and we present a modified squaring which takes advantage of the smallness of the perturbed matrix $B$ in order to reduce the number of squarings necessary.
The recursive character of the modified squarings implies only light memory requirements.
Theoretical results on local error and error propagation for splitting methods are complemented with numerical experiments and show a clear improvement over existing and highly optimized Pad\'e methods when low to medium precision is sought.


\Appendix

\section{Further approaches}
In this subsection, we collect results on approaches that are successful in the context of splittings for ordinary differential equations, however, have been found less efficient on the numerical experiments than the methods presented before.
\subsection{On processing}\label{sec.processing}
A basic property of the adjoint action,
$$
	e^{P}Ye^{-P} = e^{\mathrm{ad}_P}Y = Y + [P, Y] + \frac12 [P, [P, Y]] + \cdots
$$
together with the cheap computability of the commutator $[D,B]=DB-BD$ motivates the use of \defn{processing techniques}, well-known for the numerical integration of differential equations, to eliminate error terms.
The idea is now based on the observation that $(XYX^{-1})^N = XY^NX^{-1}$ and essentially corresponds to a change of basis in which the error propagation (recall that large $s$ can be regarded as a (long-) time integration using a small time-step $h=1/2^s$) is expected to be less severe.

The modified Strang algorithm \eqref{eq:strang62} has leading error proportional to
\[
	[B,[D,B]], \quad [B,[D,[D,[D,B]]]], \quad [D,[D,[B,[D,B]]]].
\]
The first two of which can be eliminated using a processor with $P=\alpha [D,B] + \beta [D,[D,[D,B]]]$, thus motivating the ansatz
\[
	e^{\alpha \eps h^2 [D,B] + \beta \eps  h^4 [D,[D,[D,B]]]} \tilde{Y}_{s} e^{-\alpha \eps h^2 [D,B]  - \beta \eps h^4 [D,[D,[D,B]]]}.
\]
The norm of the outer exponents is small and a low order Pad\'e approximation, say $r_2(P)$, usually provides sufficient accuracy.
Therefore, at the expense of one exponential, one multiplication and one inversion (which is performed together with the multiplication, as for the Pad\'e methods, $(\cB\cD)\tilde{\cB}^{-1}$), we get two free parameters, $\alpha,\beta$.
Using the kernel $\tilde{Y}_0$, we reach order (6,4), whereas $\tilde{Y}_1$ is sufficient for order (10,4) and (6,6,4), see Table~\ref{tab.coefs2}.

\subsection{More exponentials}
For problems where complex coefficients $a_j$ lead to a substantial increase in computational complexity (e.g., when $A,B\in\R^{n\times n}$) or matrix commutators are not desirable, it could be advantageous to allow negative values for some $b_j$.

A first example is the four-stage method
\beq\label{eq.neg4}
	S^{[4]}_4 = \cD_{ha_1} \cB_{hb_1} \cD_{ha_2} \cB_{hb_2} \cD_{ha_2} \cB_{hb_1} \cD_{ha_1}.
\eeq
This scheme requires two exponentials, two products and has two free parameters which can produce a fourth-order method with real coefficients $a_j,b_j$, known as triple jump 
  \cite{creutz89hhm,suzuki90fdo,yoshida90coh},
	see Table~\ref{tab.coefs2}.

Another product is necessary to compute the six-stage composition
\[
	S^{[6]}_{(6,4)}= \cD_{ha_1} (\cB_{hb_1} \cD_{ha_2} \cB_{hb_1})\cD_{ha_3}\cB_{hb_2} \cD_{ha_3} (\cB_{hb_1} \cD_{ha_2}\cB_{hb_1})\cD_{ha_1}.
\]
Three free parameters are sufficient to construct (6,4) methods, however, with complex time-steps.
The real-valued fourth-order method minimizing the error at $\mathcal{O}(\varepsilon h^5)$ can be found in Table~\ref{tab.coefs2}.
An additional stage with a grouping similar to the modified splittings,
\[
	S^{[7]}_{(6,4)}= \cD_{ha_1} (\cB_{hb_1} \cD_{ha_2} \cB_{hb_2}\cD_{ha_2}\cB_{hb_1}) \cD_{ha_3} 
											(\cB_{hb_1} \cD_{ha_2} \cB_{hb_2}\cD_{ha_2}\cB_{hb_1})
					 \cD_{ha_1},
\]
requires the same number of products but has real solutions of order (6,4). Among the four real-valued solutions, the one minimizing the error at $\mathcal{O}(\varepsilon h^7)$ is printed in Table~\ref{tab.coefs2}.
We have found that supposedly clever re-utilization of exponentials by setting $b_j$ to be a rational multiple of an already computed exponent $b_k$ are not competitive since - at its very best - one can save the computation of an exponential at the cost of an inversion ($b_j=-b_k$) or a matrix product ($b_j=2b_k$), however, the direct use of the sufficiently accurate $r_2$ Pad\'e method needs only one inversion.

\subsection{Splitting for low-order Pad\'e}
Technically, the stated splitting orders assume the exact computation of all exponentials, but in practice, the cheap underlying Pad\'e scheme $r_2$ has accuracy limit $\cO(\eps^3h^3)$.
Since we assumed $\eps$ to be a small parameter, comparable to $h^2$, it could be regarded as $\cO(\eps h^7)$. 
Instead of switching to the more precise $r_4$ method ($\cO(\eps^5 h^5)$) for the exponential $\cB$, 
(using $r_2$ for the processor has error $\cO(h^6 \eps^3)$ and is therefore sufficient),
we attempt to use a free parameter to decrease the $r_2$-related error in $\cB$ to $h^5\varepsilon^5$.

The procedure is based on the observation that the approximant $r_2(h\eps B)$ can be expressed as a single exponential 
$$
	r_2(h\eps B) = e^{h\eps B + h^3 \eps^3 C + \cO(h^5 \eps^5)}
$$
for some matrix $C$. 
Notice that the exponent can be expanded in odd powers of $h$ since $r_2$ is symmetric.
Now, we simply add the (unknown) matrix $C$ to the algebra and in addition to the previous order conditions, we have to solve $\sum_{i=1}^m b_i^3 = 0$.
It is clear that condition $b_i=1/m$ has to be dropped and at least three exponentials $\cB_{b_jh}$ are necessary.
We embark by modifying \eqref{eq.neg4} to 
\beq\label{eq.neg4mod}
	\Psi^{[4,mod]} = \cD_{ha_1} \tilde{\cB}_1 \cD_{ha_2} \tilde{\cB}_2 \cD_{ha_2} \tilde{\cB}_1 \cD_{ha_1}.
\eeq
Using two exponentials (inversions) and two multiplications, we have six free parameters and only one additional equation. 
The freedom in the parameters allows to construct real-coefficient methods of order (10,4)
and alternatively, at order (8,4), a method minimizing the squared error polynomials $e_{5,2}$ at $\eps^2 h^5$, see Table~\ref{tab.coefs2}.

\begin{table}[!ht]
\caption{\label{tab.coefs2}Further splitting methods, including several exponentials and processing techniques.}
\noindent{\footnotesize
\begin{tabular}{ll}
\toprule
$S^{[4]}$, 4 stages, order 4  & 2 exp, 2 prod\\[1mm]
$a_1= \frac16 (2 + 1/2^{1/3} + 2^{1/3})$, $b_1= \frac13 (2 + 1/2^{1/3} + 2^{1/3})$ & 2 complex sol. omitted
\\[1mm]
$S^{[6]}$, 6 stages, order 4   & 2 exp, 3 prod\\[1mm]
$a_1=0.19731107566242791631$,& [minimizes $\mathcal{O}(\varepsilon h^5)$]\\
$a_2=0.38252646594731312955$, \\												
$a_3=(1-2a_1-2a_2)=-0.079837541609741045862$,\\ 
$b_1=0.42519341909910345071$,\\ $b_2=1-4b_1=-0.70077367639641380284$.
\\[1mm]
$S^{[7]}$, 7 stages, order (6,4)  & 2 exp, 3 prod\\[1mm]
$a_1=0.35937529621978708941$, &{ [minimizes $\mathcal{O}(\varepsilon h^7)$]}\\
$a_2=-0.098379231055234835826$, \\												
$a_3=(1-2a_1-4a_2)=0.67476633178136516448$,\\ 
$b_1=0.67702963544760500586$,\\ $b_2=1/2-2b_1=-0.85405927089521001173$.\\
\midrule\midrule
Processed $e^{xh^2[D,B]+yh^4[D,[D,[D,B]]]}\tilde{Y}_0e^{-xh^2[D,B]-yh^[D,[D,[D,B]]]}$ & 2 exp, 1 prod, 1 inv\\[1mm]
Order (6,4)\\[1mm]
$a_1=1/2$,
$\beta=-1/24$, $\gamma=31/5760$
$x=-1/12$, $y=1/120$.\\
\midrule
Processed $e^{xh^2[D,B]+yh^4[D,[D,[D,B]]]}\tilde{Y}_1e^{-xh^2[D,B]-yh^4[D,[D,[D,B]]]}$ & 2 exp, 1 prod, 1 inv\\[1mm]
Order (6,6,4)\\[1mm]
$a_2=0.2587977340833403434530275$, \\
$\beta=-0.005227683364583625421653925$, \\
$\gamma=0.0000329546718228203782$,\\
$x=-0.02303276685416841919659022$,\\
$y=0.0007499977372301362425777840$.\\[1mm]
Order (10,4)\\[1mm]
$a_2=0.250225501288894385213924$,\\
$\beta=-0.0052083460460411565905784$,\\
$\gamma=0.0000329546718228203782$,\\
$x=-0.0208897086555569296368143$, \\
$y=0.0000573371861339342917744$\\
\midrule\midrule
$\cD_{ha_1} \tilde{\cB}_1 \cD_{ha_2} \tilde{\cB}_2 \cD_{ha_2} \tilde{\cB}_1 \cD_{ha_1}$ & 2 exp, 2 prod\\[1mm]
real (10,4) based on $r_2$, & \\
$d_2 = -0.0017987433839305087766, c_2 = -0.14389703981903926044$,\\
$d_1 = 0.000039345117326816272608, c_1 = -0.0079989398412468330564$,\\
$b_2 = -0.58268652153120735848, a_2 = 0.50468619989723192191$\\[1mm]
(8,4) minimizing $e_{5,2}$,\\ 
$d_2 = 0.009460956758445480826, c_2 = -0.03780196888453765108$, \\
$d_1 = 0.0011653151315644152329,c_1 = -0.061046475308497637733$,\\
$b_2 = -0.58268652153120735848, a_2 = 0.50468619989723192191$\\
\bottomrule
\end{tabular}
}
\end{table}

\end{document}